\documentclass[conference]{IEEEtran}
\IEEEoverridecommandlockouts
\usepackage{cite}
\usepackage{amsmath,amssymb,amsfonts,bm}
\usepackage{amsthm}
\usepackage{algorithmic}
\usepackage{graphicx}
\usepackage{textcomp}
\usepackage{xcolor}
\usepackage{cuted}
\usepackage[section]{placeins}
\usepackage{stfloats}
\usepackage{caption}
\makeatletter
  \newcommand\tabcaption{\def\@captype{table}\caption}
\makeatother
\newtheorem{theorem}{Theorem}
\newtheorem{remark}{Remark}
\newtheorem{example}{Example}
\def\BibTeX{{\rm B\kern-.05em{\sc i\kern-.025em b}\kern-.08em
    T\kern-.1667em\lower.7ex\hbox{E}\kern-.125emX}}
\begin{document}

\title{A benchmark generator for boolean quadratic programming\\
}

\author{\IEEEauthorblockN{Xiaojun Zhou}
\IEEEauthorblockA{\textit{School of Automation} \\
\textit{Central South University}\\
Changsha, China \\
410083}
\and
\IEEEauthorblockN{Tingwen Huang}
\IEEEauthorblockA{\textit{School of Engineering} \\
	\textit{Texas A\&M University at Qatar}\\
	Doha, Qatar \\
	PO Box 23874}
}

\maketitle

\begin{abstract}
For boolean quadratic programming (BQP), we will show that there is no duality gap between the primal and dual problems under some conditions by using the classical Lagrangian duality.
A benchmark generator is given to create random BQP problems which can be solved in polynomial time. Several numerical examples are generated to demonstrate the effectiveness of the proposed method.
\end{abstract}
\setlength{\parskip}{4pt}

\section{Introduction}
Boolean quadratic programming (BQP) has found many applications (see \cite{billionnet2007using} the  references therein), for example, the Boolean networks \cite{cheng2010linear}, and the well-known maxcut problem \cite{marti2009advanced}, can be formulated as BQP. Various approaches have been used to solve BQP, including semidefinite programming (SDP) relaxation \cite{poljak1995recipe}, second order cone programming (SOCP) relaxation  \cite{kim2001second}, decomposition method \cite{elloumi2000decomposition} and randomized heuristics \cite{festa2002randomized}. However, these approaches are only able to find inexact solutions to BQP. On the other hand, methods based on branch-and-cut \cite{leyffer2001integrating} can find exact solutions, but they are usually time-comsuming. 
Therefore, there exists a necessity to evaluate the performance of approaches to solving BQP. 
To provide testbed for these approaches, a benchmark generator for BQP problems is proposed in this study. 

\section{Theoretical basics}
Considering the following boolean quadratic programming (BQP) problem
\begin{equation}
\min _{\boldsymbol{x} \in \mathbb{R}^{n}}\left\{f(\boldsymbol{x})=\frac{1}{2} \boldsymbol{x}^{T} Q \boldsymbol{x}-\boldsymbol{c}^{T} \boldsymbol{x} \mid \boldsymbol{x} \in\{-1,1\}^{n}\right\},\label{eq}
\end{equation}
where, $Q=Q^{T} \in \mathbb{R}^{n \times n}$ is a given indefinite matrix, $\boldsymbol{c} \in \mathbb{R}^{n}$ is a given nonzero vector.

By introducing the Lagrange multiplier $\lambda_i$ associated with each constraint $\frac{1}{2}(x_i^2-1) = 0$, the Lagrangian $L:\mathbb{R}^n \times \mathbb{R}^n \rightarrow \mathbb{R}$ can be defined as
\begin{equation}
\begin{aligned}
L(\boldsymbol{x}, \boldsymbol{\lambda}) &=\frac{1}{2} \boldsymbol{x}^{T} Q \boldsymbol{x}-\boldsymbol{c}^{T} \boldsymbol{x}+\sum_{i=1}^{n} \frac{1}{2} \lambda_{i}\left(x_{i}^{2}-1\right) \\
&=\frac{1}{2} \boldsymbol{x}^{T} Q(\boldsymbol{\lambda}) \boldsymbol{x}-\boldsymbol{c}^{T} \boldsymbol{x}-\frac{1}{2} \boldsymbol{\lambda}^{T} \boldsymbol{e}
\end{aligned},\label{eq}
\end{equation}
where, $\boldsymbol{\lambda}=\left(\lambda_{1}, \cdots, \lambda_{n}\right)^{T}, \boldsymbol{e} \in \mathbb{R}^{n}$ is a vector with all entries one, and $Q_{\lambda}$ is
\begin{equation}
Q(\boldsymbol{\lambda})=Q+\operatorname{diag}(\boldsymbol{\lambda}),\label{eq}
\end{equation}
here, $\operatorname{diag}(\boldsymbol{\lambda})$ is a diagonal matrix with its diagonal entries $\lambda_{1}, \cdots, \lambda_{n}$.

The Lagrangian dual function can be obtained by
\begin{equation}
g(\boldsymbol{\lambda})=\inf _{\boldsymbol{x} \in\{-1,1\}^{n}} L(\boldsymbol{x}, \boldsymbol{\lambda}).\label{eq}
\end{equation}

Let define the following dual feasible space
\begin{equation}
\mathcal{S}^{+}=\left\{\boldsymbol{\lambda} \in \mathbb{R}^{n} \mid Q(\boldsymbol{\lambda}) \succ 0\right\},\label{eq}
\end{equation}
then the Lagrangian dual function can be written explicitly as
\begin{equation}
g(\boldsymbol{\lambda})=-\frac{1}{2} \boldsymbol{c}^{T} Q^{-1}(\boldsymbol{\lambda}) \boldsymbol{c}-\frac{1}{2} \boldsymbol{\lambda}^{T} \boldsymbol{e}, \text { s.t. } \boldsymbol{\lambda} \in \mathcal{S}^{+},\label{eq}
\end{equation}
associated with the Lagrangian equation
\begin{equation}
Q(\boldsymbol{\lambda}) \boldsymbol{x}=\boldsymbol{c}.\label{eq}
\end{equation}

Finally, the Lagrangian dual problem can be obtained as
\begin{equation}
\max _{\boldsymbol{\lambda} \in \mathcal{S}^{+}}\left\{g(\boldsymbol{\lambda})=-\frac{1}{2} \boldsymbol{c}^{T} Q^{-1}(\boldsymbol{\lambda}) \boldsymbol{c}-\frac{1}{2} \boldsymbol{\lambda}^{T} \boldsymbol{e}\right\}.\label{eq}
\end{equation}

\begin{theorem}\label{the1}
If $\overline{\boldsymbol{\lambda}}$ is a critical point of the Lagrangian dual function and $\lambda \in S^{+}$, then the corresponding
$\overline{\boldsymbol{x}}=Q^{-1}(\overline{\boldsymbol{\lambda}}) \boldsymbol{c}$ is a global solution to the BQP problem.
\end{theorem}
\begin{proof}
The derivative of $g(\boldsymbol{\lambda})$ gives
\begin{equation*}
\begin{aligned}
\frac{\partial g(\boldsymbol{\lambda})}{\partial \boldsymbol{\lambda}_{i}} &=-\frac{1}{2} \frac{\partial\left[\boldsymbol{c}^{T} Q^{-1}(\boldsymbol{\lambda}) \boldsymbol{c}\right]}{\partial \boldsymbol{\lambda}_{i}}-\frac{1}{2} \\
&=\frac{1}{2} \boldsymbol{c}^{T} Q^{-1}(\overline{\boldsymbol{\lambda}}) \frac{\partial Q(\boldsymbol{\lambda})}{\partial \overline{\boldsymbol{\lambda}}_{i}} Q^{-1}(\overline{\boldsymbol{\lambda}}) \boldsymbol{c}-\frac{1}{2} \\
&=\frac{1}{2} \overline{\boldsymbol{x}}_{i}^{2}-\frac{1}{2}
\end{aligned}
\end{equation*}
where, $\overline{\boldsymbol{x}}=Q^{-1}(\overline{\boldsymbol{\lambda}}) \boldsymbol{c}$. Since $\overline{\boldsymbol{\lambda}}$ is a critical point of the Lagrangian dual function, we have $\frac{1}{2} \overline{\boldsymbol{x}}_{i}^{2}-\frac{1}{2}=0$.

To continue, we have
\begin{eqnarray*}
&&f(\bm x) \!-\! f(\bar{\bm x}) \\
&=& \frac{1}{2} \bm x^T Q \bm x - \bm c^T \bm x - \frac{1}{2} \bar{\bm x}^T Q \bar{\bm x} + \bm c^T \bar{\bm x} - \sum_{i=1}^n \frac{1}{2} \bar{\lambda}_i (\bar{x}_i^2 - 1) \\
&=& \frac{1}{2} (\bm x - \bar{\bm x})^T Q(\bar{\bm \lambda})(\bm x - \bar{\bm x}) \!-\! \frac{1}{2} \bm x^T \mathrm{diag}(\bar{\bm \lambda}) \bm x \!+\! \bm x^T Q(\bar{\bm \lambda}) \bar{\bm x}\\
&& \!-\!
\bar{\bm x}^T Q(\bar{\bm \lambda}) \bar{\bm x}  - \bm c^T (\bm x - \bar{\bm x}) + \frac{1}{2}\bar{\bm \lambda}^T \bm e \\
&=& \frac{1}{2} (\bm x - \bar{\bm x})^T Q(\bar{\bm \lambda})(\bm x - \bar{\bm x}) + (\bm x - \bar{\bm x})^T(Q(\bar{\bm \lambda}) \bar{\bm x} - \bm c)\\
&=&\frac{1}{2} (\bm x - \bar{\bm x})^T Q(\bar{\bm \lambda})(\bm x - \bar{\bm x}) > 0, \forall \bm x \in \{-1,1\}^n
\end{eqnarray*}
that is to say, $\overline{\boldsymbol{x}}=Q^{-1}(\overline{\boldsymbol{\lambda}}) \boldsymbol{c}$ is the global minimizer of the BQP problem.
This completes the proof.
\end{proof}
\begin{remark}
Under the conditions in Theorem \ref{the1}, it is easy to verify that
\begin{eqnarray*}
\begin{aligned}
g(\overline{\boldsymbol{\lambda}}) &=-\frac{1}{2} \boldsymbol{c}^{T} Q^{-1}(\overline{\boldsymbol{\lambda}}) \boldsymbol{c}-\frac{1}{2} \overline{\boldsymbol{\lambda}}^{T} \boldsymbol{e} \\
&=-\frac{1}{2} \overline{\boldsymbol{x}}^{T} Q(\boldsymbol{\lambda}) \overline{\boldsymbol{x}}+\left(\overline{\boldsymbol{x}}^{T} Q(\boldsymbol{\lambda}) \overline{\boldsymbol{x}}-\boldsymbol{c}^{T} \overline{\boldsymbol{x}}\right)-\frac{1}{2} \boldsymbol{\lambda}^{T} \boldsymbol{e} \\
&=\frac{1}{2} \overline{\boldsymbol{x}}^{T} Q(\overline{\boldsymbol{\lambda}}) \overline{\boldsymbol{x}}-\boldsymbol{c}^{T} \overline{\boldsymbol{x}}-\frac{1}{2} \overline{\boldsymbol{\lambda}}^{T} \boldsymbol{e}=L(\overline{\boldsymbol{x}}, \overline{\boldsymbol{\lambda}}) \\
&=\frac{1}{2} \overline{\boldsymbol{x}}^{T} Q \overline{\boldsymbol{x}}-\boldsymbol{c}^{T} \overline{\boldsymbol{x}}+\sum_{i=1}^{n} \frac{1}{2} \bar{\lambda}_{i}\left(\bar{x}_{i}^{2}-1\right) \\
&=f(\overline{\boldsymbol{x}}),
\end{aligned}
\end{eqnarray*}
which shows that there is no duality gap between the primal and dual problems.
\end{remark}

\section{A benchmark generator for BQP}
As stated above, if there exists a critical point of Lagrangian dual function in $\mathcal{S}^{+}$, zero duality gap will be met.
In this section, we will construct a benchmark generator for BQP problem such that these conditions are satisfied.

The inverse problem can be simplified as follows
\begin{eqnarray}
\begin{array}{cl}
\text { find } & Q, c, x, \lambda \\
\text { s.t. } & (Q+\operatorname{diag}(\lambda)) x=c \\
& Q+\operatorname{diag}(\lambda) \succ 0 \\
& x \in\{-1,1\}^{n} \label{eq}
\end{array}
\end{eqnarray}
where, $Q=Q^{T} \in \mathbb{R}^{n \times n}, \boldsymbol{c}, \boldsymbol{x}, \boldsymbol{\lambda} \in \mathbb{R}^{n}$.

Suppose $Q$ is a freely random symmetric matrix, to guarantee $Q+\operatorname{diag}(\boldsymbol{\lambda}) \succ 0$, let $\lambda$ satisfy
\begin{eqnarray}
\lambda_{i} \geq \sum_{j=1}\left|Q_{i j}\right|,\label{eq}
\end{eqnarray}
to make sure that $Q(\lambda)$ is a diagonally dominant matrix.

Suppose $\boldsymbol{x}$ is a freely random ``true" solution, then $\boldsymbol{c}$ should satisfy
\begin{eqnarray}
c=(Q+\operatorname{diag}(\boldsymbol{\lambda})) \boldsymbol{x}.\label{eq}
\end{eqnarray}

The matlab scripts for generating a benchmark BQP are given in the following
\begin{verbatim}
function [Q,c,x,lambda] = generate_Qc(n)
base = 10;
Q = base*randn(n);
Q = round((Q + Q')/2);
lambda = zeros(n,1);
x = round(rand(n,1));
x = 2*x - 1;
lambda = sum(abs(Q),2);
c = (Q + diag(lambda))*x;
\end{verbatim}
where, \textit{base} is set to control the range of elements in $Q$.\\
\section{Numerical experiments}
The Lagrangian dual problem can be reformulated as the following SDP problem easily via Schur complement \cite{boyd1994linear}
\begin{eqnarray}
\begin{aligned}
(\mathrm{SDP}): & \min \frac{1}{2} t+\frac{1}{2} \lambda^{T} e \\
& \text { s.t. }\left(\begin{array}{c}
Q(\lambda) \quad c \\
c^{T} \qquad t
\end{array}\right) \succeq 0 \label{eq}
\end{aligned}
\end{eqnarray}

In the next, several examples are created by the above mentioned generator function \textit{generate\_Qc}$(\cdot)$.
All of the experiments are run on MATLAB R2010b on Intel(R) Core(TM) i3-2310M CPU @2.10GHz under Window 7 environment.
The SDPT3 \cite{toh1999sdpt3} is used as a solver embedded in YALMIP \cite{lofberg2004yalmip} for the SDP problem.
\begin{example}\label{exa1}
\begin{eqnarray*}
Q=\left(
 \begin{array}{ccccc}
   -4 & -3  & 6   & -3 & -6 \\
   -3 & 13  & -25 & -5 & 3 \\
   6  & -25 & -2  & -5 & -1 \\
   -3 & -5  & -5  & -8 & -7 \\
   -6 & 3   & -1  & -7 & -5
 \end{array}
\right),
\quad
\boldsymbol{c}=\left(
 \begin{array}{c}
  -18 \\
  92 \\
  -62 \\
  -10 \\
  0
 \end{array}
\right)
\end{eqnarray*}
\end{example}
By solving the corresponding Lagrangian dual problem in 0.9204 seconds, we can get $\overline{\boldsymbol{\lambda}}=(21.9996,48.9999,39.0000,27.9996,21.9998)^{T}$ and $\overline{\boldsymbol{x}}=Q^{-1}(\overline{\boldsymbol{\lambda}})c$ =$(-1.0000,1.0000, -1.0000,-1.0000,-1.0000)^T$.

\begin{strip}
\begin{example}\label{exa2}
 \begin{equation*}
	\boldsymbol{Q}=
      \begin{pmatrix}
        -6  &  -9  &  -7  &  -3   &  2  &  -2  &  -12 & -11  &    8 &  -6 \\
        -9  &  14  &   -4 &   3   &  -2 &  -4  &   6  &  23  &   8  &  3 \\
        -7  &  -4  &   1  &  21   & -10 &    2 &    6 &  -13 &  -9  &  4 \\
        -3  &   3  &  21  & -18   &  -2 &    7 &   -2 &   16 &   1  &   3 \\
        2   &  -2  & -10  &  -2   & 10  &   8  & -11  &  -3  &  -2  &  -2 \\
        -2  &  -4  &   2  &   7   &  8  &   0  & -10  & -10  & -2   &   -7 \\
        -12 &   6  &   6  &  -2   & -11 &  -10 &   -7 &  -10 &   3  &   0 \\
        -11 &  23  & -13  &  16   & -3  & -10  & -10  &   6  &   -8 &  10 \\
        8   &   8  &  -9  &   1   & -2  &  -2  &   3  &  -8  &  -10 &   -5\\
        -6  &   3  &   4  &   3   & -2  &  -7  &   0  &  10  &   -5 &  -9 \\
     \end{pmatrix},
     \quad
     \boldsymbol{c}=\left(
     \begin{array}{c}
     24\\
     -84\\
     72\\
     -18\\
     -72\\
     16\\
     38\\
     54\\
     -66\\
     42\\
   \end{array}
   \right)
 \end{equation*}
\end{example}
\end{strip}

By solving the Lagrangian dual problem  in 0.9204 seconds, we can get
$\overline{\boldsymbol{\lambda}}
= (66.0010,75.9997,76.9970,\newline 76.0010,51.9993,51.9988,66.9988,109.9979, 55.9985,\newline 48.9995)^{T}$
and  $\overline{\boldsymbol{x}}=Q^{-1}(\overline{\boldsymbol{\lambda}})c= (1.0000,-1.0000, 1.0000,\newline -1.0000, -1.0000, 1.0000,1.0000,1.0000,-1.0000,1.0000)^{T}$.

\begin{strip}
\begin{example}\label{exa3}
\setcounter{MaxMatrixCols}{20}
 \begin{eqnarray*}
	\boldsymbol{Q}=
     \begin{pmatrix}
      11  &    2  & -10  &  15  &  21   &  -7   &  -8  &   6  &   -3  &   11  &  2   &  -1  &    3  &  -2  &  -1\\
     2   &   6   &   7  &   -5 &   10  &  -14  & -1   &  -8  &   3   &   6   &   6  &  0   &   -7  &  1   &  -2\\
    -10  &   7   &  21  &   12 &   13  &   -9  &  -1  &   2  &  -5   &   9   &   2  &  -1  &   -2  &  4   &  8\\
    15   &  -5   &   12 &   12 &   -7  &   0   & -3   & -17  &   -3  &   6   &  -1  &  -1  &   -1  &  6   &  -5\\
    21   &  10   &  13  &   -7 &   3   &   6   &   3  &  -1  &  -10  &   0   &  -9  &  -1  &   -4  &  -7  &  -2\\
    -7   &  -14  &  -9  &   0  &   6   &   5   &   1  &   7  &   3   &   2   &  -1  &   3  &   -4  &  3   &  8\\
    -8   &  -1   &  -1  &  -3  &   3   &   1   &  -5  &  -5  &   3   &   6   &   17 & -13  &   6   &  14  & -10\\
     6   &  -8   &   2  &  -17 &   -1  &   7   & -5   &  -2  &   -6  &   1   &  12  &   0  &   5   &   5  &  -4\\
    -3   &  3    &  -5  &  -3  & -10   &   3   &   3  &  -6  &  -26  &   3   &   -4 &   7  &  13   &  -4  &  -2\\
    11   &   6   &   9  &   6  &   0   &   2   &   6  &   1  &   3   &  13   &  -11 &  10  &  -12  & -13  &  -11\\
     2   &   6   &   2  &  -1  &   -9  &   -1  &  17  &  12  &  -4   &  -11  &  -1  &  12  &  -9   &   7  &  5\\
    -1   &   0   &  -1  &   -1 &  -1   &   3   & -13  &   0  &   7   &   10  &   12 &  -1  &  -1   &   5  &   9\\
     3   &   -7  &  -2  &   -1 &  -4   &  -4   &   6  &  5   &  13   &  -12  &   -9 &  -1  & -15   &   6  &  -1\\
    -2   &    1  &   4  &   6  & -7    &   3   &  14  &   5  &  -4   &  -13  &   7  &   5  &   6   &  15  &  10\\
    -1   &   -2  &  8   &  -5  &  -2   &   8   &  -10 &  -4  &  -2   &  -11  &   5  &   9  &  -1   &   10 &   8\\
     \end{pmatrix},
     \quad
     \boldsymbol{c}=\left(
     \begin{array}{c}
     -140\\
      86 \\
     -118 \\
     -114 \\
     -134 \\
     -72  \\
     92  \\
     -100 \\
     120 \\
     98  \\
     -80 \\
     70  \\
     90  \\
     120 \\
      78 \\
   \end{array}
   \right)
 \end{eqnarray*}
\end{example}
\setlength{\parskip}{4pt}
\captionsetup[table]{labelformat=simple ,labelsep=quad}
\renewcommand{\thetable}{\arabic{table}}
\end{strip}

By solving the Lagrangian dual problem  in 0.9204 seconds, we can get
$\overline{\boldsymbol{\lambda}}
= (102.9966,77.9974,105.9976,\newline  93.9972,96.9967,72.9976,95.9974, 80.9970, 94.9967,\newline113.9981,98.9980,64.9976,88.9972,101.9973, 85.9979)^{T}$
and $\overline{\boldsymbol{x}}=Q^{-1}(\overline{\boldsymbol{\lambda}})c= (-1.0000,1.0000,-1.0000,\newline -1.0000,-1.0000,-1.0000,1.0000,-1.0000, 1.0000,1.0000,\newline -1.0000, 1.0000, 1.0000,1.0000,1.0000)^{T}$.

\begin{remark}
	All of the tested problems are solved in 0.9204 seconds, which can be regarded as an indicator for polynomial time complexity.
\end{remark}

\begin{example}
\textup{Other large random BQP problems are created by the same procedure, and the corresponding running time for solving these problems can be found in Fig. \ref{fig_runtime}.}
\end{example}

\begin{figure}[htbp]
	\centerline{\includegraphics[width=8cm,height=6cm]{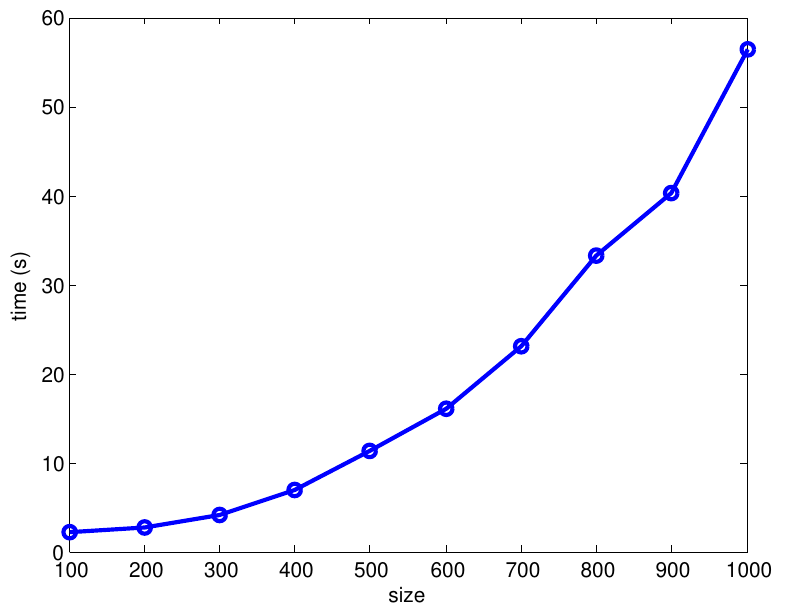}}
	\caption{Running time for other large random BQP cases}
	\label{fig_runtime}
\end{figure}

\section{Conclusion}
In this short note, a benchmark generator for boolean quadratic programming (BQP) was implemented in Matlab. And these BQP instances can be used for evaluating the performance of various approaches.


\begin{thebibliography}{00}
\bibitem{billionnet2007using}
Billionnet A, Elloumi S. Using a mixed integer quadratic programming solver for the unconstrained quadratic 0-1 problem. Mathematical Programming, 109(1): 55--68, 2007.
\bibitem{cheng2010linear}
Cheng D, Qi H. A linear representation of dynamics of Boolean networks. IEEE Transactions on Automatic Control, 55(10): 2251-2258, 2010.
\bibitem{marti2009advanced}
Marti R, Duarte A, Laguna M. Advanced scatter search for the max-cut problem. INFORMS Journal on Computing, 21(1): 26-38, 2009.
\bibitem{poljak1995recipe}
Poljak S, Rendl F, Wolkowicz H. A recipe for semidefinite relaxation for (0, 1)-quadratic programming. Journal of Global Optimization, 7(1): 51-73, 1995.
\bibitem{kim2001second}
Kim S, Kojima M. Second order cone programming relaxation of nonconvex quadratic optimization problems. Optimization Methods and Software, 15(3-4): 201--224,  2001.
\bibitem{elloumi2000decomposition}
Elloumi S, Faye A, Soutif E. Decomposition and linearization for 0-1 quadratic programming. Annals of Operations Research, 99(1-4): 79-93, 2000.
\bibitem{festa2002randomized}
Festa P, Pardalos P M, Resende M G C, et al. Randomized heuristics for the MAX-CUT problem. Optimization methods and software, 17(6): 1033--1058, 2002.
\bibitem{leyffer2001integrating}
Leyffer S. Integrating SQP and branch-and-bound for mixed integer nonlinear programming[J]. Computational optimization and applications, 18(3): 295--309, 2001.

\bibitem{boyd1994linear} Boyd S P, Ghaoui L EI, Feron E and Balakrishnan V, Linear matrix inequalities in system and control theory, Philadelphia: SIAM, 1994.
\bibitem{toh1999sdpt3} Toh K C, Todd M J and T{\"u}t{\"u}nc{\"u} R H, SDPT3-a MATLAB software package for semidefinite programming, version 1.3, Optimization methods and software, 11:545--581, 1999.
\bibitem{lofberg2004yalmip} Lofberg J, YALMIP: A toolbox for modeling and optimization in MATLAB, IEEE International Symposium on Computer Aided Control Systems Design, pp. 284--289, 2004.

\end{thebibliography}
\end{document}